\documentclass[a4paper,11pt,reqno]{amsart}
\usepackage{marginnote,amsmath,amsfonts,amscd,amssymb,graphicx,mathrsfs,eufrak,dsfont}

\usepackage[dvips,all,arc,curve,color,frame]{xy}
\usepackage[usenames]{color}
\usepackage{enumitem}

\usepackage[colorlinks]{hyperref}
\usepackage{tikz,mathrsfs}
\usetikzlibrary{arrows,decorations.pathmorphing,decorations.pathreplacing,positioning,shapes.geometric,shapes.misc,decorations.markings,decorations.fractals,calc,patterns}

\newtheorem{thm}{Theorem}[section]

\newtheorem{lemma}[thm]{Lemma}

\newtheorem{cor}[thm]{Corollary}

\newtheorem{obs}[thm]{Observation}

\theoremstyle{definition}

\begin{document}
\title{A Note on Restricted Online Ramsey Numbers of Matchings}

\author{Vojt\v{e}ch Dvo\v{r}\'ak}
\address[Vojt\v{e}ch Dvo\v{r}\'ak]{Department of Pure Mathematics and Mathematical Statistics, University of Cambridge, UK}
\email[Vojt\v{e}ch Dvo\v{r}\'ak]{vd273@cam.ac.uk}

\begin{abstract}
The restricted online Ramsey numbers were introduced by Conlon, Fox, Grinshpun and He \cite{conlon} in 2019. In a recent paper \cite{briggs}, Briggs and Cox studied the restricted online Ramsey numbers of matchings and determined a general upper bound for them. They proved that for $n=3r-1=R_2(r K_2)$ we have $\tilde{R}_{2}(r K_2;n) \leq n-1$ and asked whether this was tight. In this short note, we provide a general lower bound for these Ramsey numbers. As a corollary, we answer this question of Briggs and Cox, and confirm that for $n=3r-1$ we have $\tilde{R}_{2}(r K_2;n) = n-1$. We also show that for $n'=4r-2=R_3(r K_2)$ we have $\tilde{R}_{3}(r K_2;n') = 5r-4$. 
\end{abstract}

\maketitle
\parindent 20pt
\parskip 0pt

\section{Introduction}\label{intro}

For families of graphs $\mathcal{G}_{1},...,\mathcal{G}_t$, the Ramsey number $R(\mathcal{G}_1,...,\mathcal{G}_t)$ is the smallest integer $n$ such that any colouring of $K_n$ with colours $1,...,t$ contains a graph $G_i$ in colour $i$ for some $G_i \in \mathcal{G}_i$ and some $i\in \left\{ 1,...,t \right\}$. The Ramsey numbers of graphs have been studied extensively, see for instance a survey of Conlon, Fox and Sudakov \cite{conlonsurvey}. 

Many variants of the Ramsey numbers have been considered. In particular, in 2019, Conlon, Fox, Grinshpun and He \cite{conlon} introduced the so-called restricted online Ramsey numbers. For families of graphs $\mathcal{G}_{1},...,\mathcal{G}_t$ and integer $n$ such that $n \geq R(\mathcal{G}_1,...,\mathcal{G}_t)$, the restricted online Ramsey number $\tilde{R}(\mathcal{G}_{1},...,\mathcal{G}_t;n)$ is the smallest integer $k$ for which Builder can always guarantee a win within the first $k$ moves of the following game between Builder and Painter. In each turn, Builder picks an edge of initially uncoloured $K_n$ and Painter chooses any colour out of $1,...,t$ and colours the edge with this colour. Builder wins once there is a graph $G_i$ in colour $i$ for some $G_i \in \mathcal{G}_i$ and some $i\in \left\{ 1,...,t \right\}$. We note that the definitions of $\tilde{R}(\mathcal{G}_{1},...,\mathcal{G}_t;n)$ differ slightly between the previous papers on this topic \cite{briggs,conlon,gonzalez}, but it is easy to see that all are equivalent.

Briggs and Cox \cite{briggs} studied the restricted online Ramsey numbers of matchings and trees. They proved the following general theorem.

\begin{thm}\label{briggsmain}
Fix $t \geq 2$ and positive integers $r_1,...,r_t$. If $n \geq R(r_1 K_1,...,r_t K_t)$, then

\begin{equation*}
\tilde{R}(r_1 K_1,...,r_t K_t;n) \leq \frac{2t-1+(t-3) \log_{2}(t-2)}{t+1}n
\end{equation*}
\noindent
with the convention that $\log_{2}0=0$.
\end{thm}

Briggs and Cox \cite{briggs} describe a further refinements of their proof in the cases $t=2,3,4$, which imply the following result for $r_1=...=r_t=r$ and $n=R_t(r K_2)$.

\begin{thm}\label{briggsspecific} 
Fix $r \geq 1$ and let $n_2=R_2(r K_2)=3r-1$, $n_3=R_3(r K_2)=4r-2$, $n_4=R_4(r K_2)=5r-3$. Then $\tilde{R}_{2}(r K_2;n_2) \leq 3r-2=n_2 -1$, $\tilde{R}_{3}(r K_2;n_3) \leq 5r-4$ and $\tilde{R}_{4}(r K_2;n_4) \leq 7r-5$.
\end{thm}
 
They ask whether we have $\tilde{R}_{2}(r K_2;n_2) = n_2 -1$. The aim of this short note is to verify that this indeed holds. We also show that the bound $\tilde{R}_{3}(r K_2;n_3) \leq 5r-4$ is tight and that the bound $\tilde{R}_{4}(r K_2;n_4) \leq 7r-5$ is tight except possibly for the exact value of the additive constant.

By describing a suitable strategy of Painter, we prove the following more general lower bound.

\begin{thm}\label{main}
Fix $t \geq 2$ and positive integers $r_1,...,r_t$. If $n \geq R(r_1 K_1,...,r_t K_t)$, then $\tilde{R} (r_1 K_2,...,r_t K_2;n) \geq 3(\sum_{i=1}^{t}r_i-t+1)-n$.
\end{thm}

As a corollary, we answer the question of Briggs and Cox \cite{briggs}.

\begin{cor}
Fix $r \geq 1$ and let $n_2=R_2(r K_2)=3r-1$, $n_3=R_3(r K_2)=4r-2$, $n_4=R_4(r K_2)=5r-3$. Then $\tilde{R}_{2}(r K_2;n_2) = 3r-2=n_2 -1$, $\tilde{R}_{3}(r K_2;n_3) = 5r-4$ and $\tilde{R}_{4}(r K_2;n_4) \in \left\{ 7r-6,7r-5 \right\} $.
\end{cor}

It remains unclear whether for $t$ and $r$ large and $n=R_t(r K_2)$, the magnitude of $\tilde{R}_{t}(r K_2 ;n)$ is closer to the upper bound from Theorem \ref{briggsmain} or to the lower bound from Theorem \ref{main}.

\section{Proof of Theorem \ref{main} }

Consider the game played with $t$ colours on the edges of an initially uncoloured $K_n$. To prove Theorem \ref{main}, we will describe a strategy of Painter that ensures that after $T=3(\sum_{i=1}^{t}r_i-t+1)-n-1$ moves (where by a move we mean Builder choosing some still uncoloured edge and Painter colouring it), there is no $r_i K_2$ of colour $i$ for $i=1,...,t$.

While taking her turns (and to help her with her colouring decisions), Painter will moreover assign the following states to the coloured edges of $K_n$ and to all the vertices of $K_n$. Coloured edges are either \textit{free}, or \textit{rooted}. Every rooted edge is characterized by its \textit{root}, which is a vertex of $K_n$. Painter will assign (and update) the states of the coloured edges according to the strategy described below.

Vertices are of three types, characterized in the following way.

\begin{itemize}
\item If a vertex $v$ is a root of at least one coloured edge, it is of type I.

\item If a vertex $v$ is not of type I, but there is at least one free edge with endpoint $v$, it is of type II.

\item If a vertex $v$ is neither of type I nor of type II, it is of type III.
\end{itemize}

In particular, note that initially all the vertices are of type III, since no edges are coloured at the start of the game.

For $0 \leq j \leq $ $n \choose 2$ and $i=1,...,t$, let $A_j(i)$ be a number of type I vertices that are roots to at least one edge of colour $i$ after $j$ moves and let $B_j(i)$ be a number of free edges of colour $i$ after $j$ moves. Let $A_j=\sum_{i=1}^{t}A_j(i)$ and $B_j=\sum_{i=1}^{t}B_j(i)$.

Assume Builder chooses the edge $ab$ in $(k+1)$st turn of hers (where $0 \leq k \leq $ $n \choose 2$ $-1$). Without loss of generality (as we could otherwise switch $a$ and $b$), we can assume that if $b$ is of type I, then $a$ is also of type I; and if $b$ is of type II, then $a$ is of type I or of type II. Painter chooses the colour of an edge and updates the states of the coloured edges as follows.

\begin{enumerate}[label=(\roman*)]
\item\label{i} If $a$ is a vertex of type I, we declare the edge $ab$ to be rooted at $a$. By definition, there exists at least one other edge rooted at $a$, of some colour $c_1$ (if there are more edges rooted at $a$, pick one arbitrarily). We colour $ab$ by colour $c_1$.

\item\label{ii} If $a$ is a vertex of type II, there exists by definition a free edge $ac$ for some $c$, of some colour $c_2$ (if there are more free edges with endpoint $a$, pick one arbitrarily). We declare both edges $ab,ac$ to be rooted at $a$ and colour $ab$ in $c_2$.

\item\label{iii} If $a$ is a vertex of type III, then the edge $ab$ is declared to be free. It is coloured in any colour $c_3$ such that $A_k(c_3)+B_k(c_3) \leq r_{c_3}-2$ if at least one such colour exists, and if not in an arbitrary colour.
\end{enumerate}

The next two observations are straightforward.

\begin{obs}\label{typeiii}
The number of vertices of type III:
\begin{itemize}
\item stays the same during move \ref{i}

\item increases by $1$ or stays the same during move \ref{ii}

\item decreases by $2$ during move \ref{iii}
\end{itemize}
\end{obs}

\begin{obs}\label{sumab}
If move $j$ was $\ref{i}$ or $\ref{ii}$, we have $A_j(i)+B_j(i)=A_{j-1}(i)+B_{j-1}(i)$ for $i=1,...,t$. If move $j$ was $\ref{iii}$ and Painter used colour $c$, we have $A_j(c)+B_j(c)=A_{j-1}(c)+B_{j-1}(c)+1$ and for any $c' \neq c$ we have $A_j(c')+B_j(c')=A_{j-1}(c')+B_{j-1}(c')$.
\end{obs}

Using Observation \ref{typeiii} and Observation \ref{sumab}, we prove the key lemma. 

\begin{lemma}\label{useful}
We have $A_T+B_T \leq \sum_{i=1}^{t}r_i -t$.
\end{lemma}

\begin{proof}
Let $C_2$ be the number of moves \ref{ii} up to time $T$, and let $C_3$ be the number of moves \ref{iii} up to time $T$. At time $T$, by Observation \ref{typeiii} we have at most $n+C_2-2 C_3$ vertices of type III. That implies $n+C_2-2 C_3 \geq 0$. Since we further have $C_2+C_3 \leq T$, we must have $C_3 \leq \frac{n+T}{3}$.

Now by Observation \ref{sumab}, $A_T+B_T \leq C_3 \leq \frac{n+T}{3}=\sum_{i=1}^{t}r_i-t+\frac{2}{3}$, and since $A_T+B_T$ is an integer, we have $A_T+B_T \leq \sum_{i=1}^{t}r_i-t$ as required.
\end{proof}

Continuing the proof of Theorem \ref{main}, we are now ready to show that after $T$ moves, there is no $r_i K_2$ of colour $i$ for $i=1,...,t$.

Note that the existence of $r_m K_2$ of colour $m$ would in particular imply that $A_T(m)+B_T(m) \geq r_m$. Because of the strategy of Painter and Observation \ref{sumab}, that would imply that $A_T(i)+B_T(i) \geq r_i -1$ for $i=1,...,t$. Hence we would have $A_T+B_T \geq (r_1-1)+...+r_m+...+(r_t-1)= \sum_{i=1}^{t}r_i-t+1$, contradicting Lemma \ref{useful}. Thus the proof of Theorem \ref{main} is finished.

\section*{Acknowledgements}

The author would like to thank his PhD supervisor professor B\'{e}la Bollob\'{a}s for his support.


\begin{thebibliography}{VdB}

\bibitem[1]{briggs}
J. Briggs, C. Cox \emph{Restricted online Ramsey numbers of matchings and trees}, Electronic Journal of Combinatorics 27(3)(2020), P3.49.


\bibitem[2]{conlon}
D. Conlon, J. Fox, A. Grinshpun, X. He \emph{Online Ramsey numbers and the subgraph query problem}, In Building Bridges II (2019), 159--164.

\bibitem[3]{conlonsurvey}

D. Conlon, J. Fox, B. Sudakov \emph{Recent developments in graph Ramsey theory},  In A. Czumaj, A. Georgakopoulos, D. Král, V. Lozin, O. Pikhurko (Eds.), Surveys in Combinatorics (2015) 49--118. 

\bibitem[4]{gonzalez}
D. Gonzalez, X. He, H. Zheng \emph{An upper bound for the restricted online Ramsey
number}, Discrete Mathematics 342(9) (2019), 2564–-2569.


\end{thebibliography}
\end{document}